\pdfoutput=1 

\documentclass[a4paper,oneside,leqno,10pt,final]{article}

\usepackage{a4wide}
\usepackage[utf8]{inputenc}
\usepackage[a4paper]{geometry}

\usepackage{titlesec}

\usepackage{bm}

\usepackage{enumitem}
\setlist{itemsep=3pt,parsep=1pt,topsep=3pt,partopsep=1pt}
\setlist[enumerate,1]{label=(\arabic*)}
\setlist[enumerate,2]{label=(\alph*)}

\usepackage[dvipsnames]{xcolor}

\usepackage{graphicx}

\usepackage{mathtools}
\mathtoolsset{showonlyrefs}

\usepackage[only,llbracket,rrbracket]{stmaryrd}

\usepackage[T1]{fontenc}

\usepackage[p,osf]{scholax}
\usepackage{amsmath,amsthm}
\usepackage[scaled=1.075,ncf,vvarbb]{newtxmath}

\usepackage{chngcntr}

\usepackage{url}
\usepackage[colorlinks=true,linkcolor=blue,citecolor=Violet,linktocpage=false]{hyperref}

\swapnumbers

\theoremstyle{plain}
\newtheorem{theorem}{Theorem}[section]
\newtheorem{lemma}[theorem]{Lemma}
\newtheorem{corollary}[theorem]{Corollary}

\newtheorem*{theorem*}{Theorem}
\newtheorem*{corollary*}{Corollary}

\theoremstyle{definition}
\newtheorem{definition}[theorem]{Definition}

\newtheorem{miniremark}[theorem]{}

\newtheorem*{notation*}{Notation}

\theoremstyle{remark}

\swapnumbers

\newtheoremstyle{newremark}  
{}                           
{}                           
{}                           
{}                           
{}                           
{.}                          
{ }                          
{{\bfseries \thmnumber{#2}}{\itshape \thmname{ #1}}\thmnote{ (#3)}} 

\theoremstyle{newremark}
\newtheorem{remark}[theorem]{Remark}

\newcommand{\End}[1]{ \mathrm{End}({#1}) }

\newcommand{\ccspace}[1]{\mathscr{K}(#1)}

\newcommand{\grass}[2]{\mathbf{G}(#1,#2)}

\newcommand{\Var}[1]{\mathbf{V}_{#1}}     
\newcommand{\IVar}[1]{\mathbf{IV}_{#1}}   
     
\newcommand{\var}[1]{\mathbf{v}_{#1}}     

\DeclareMathOperator{\graph}{graph}     

\newcommand{\adim}{n}
\newcommand{\vdim}{k}
\newcommand{\codim}{n-k}

\newcommand{\oball}[2]{\mathbf{U}(#1,#2)}

\newcommand{\sphere}[1]{\mathbb{S}^{#1}}

\ifdefined\textint
    \renewcommand{\textint}[2]{{\textstyle\int_{#1}^{#2}}}
\else    
    \newcommand{\textint}[2]{{\textstyle\int_{#1}^{#2}}}
\fi

\ifdefined\textfint
    \renewcommand{\textfint}[2]{{\textstyle\fint_{#1}^{#2}}}
\else    
    \newcommand{\textfint}[2]{{\textstyle\fint_{#1}^{#2}}}
\fi

\ifdefined\textsum
    \renewcommand{\textsum}[2]{{\textstyle\sum_{#1}^{#2}}}
\else  
    \newcommand{\textsum}[2]{{\textstyle\sum_{#1}^{#2}}}
\fi

\ifdefined\textprod
  \renewcommand{\textprod}[2]{{\textstyle\prod_{#1}^{#2}}}
\else  
  \newcommand{\textprod}[2]{{\textstyle\prod_{#1}^{#2}}}
\fi

\newcommand{\natp}{\mathscr{P}}
\newcommand{\nat}{\natp \cup \{0\}}

\newcommand{\integers}{\mathbf{Z}}

\newcommand{\R}{\mathbf{R}}

\newcommand{\CF}[1]{ \raisebox{\depth}{$\boldsymbol{\chi}$}_{#1} }

\newcommand{\LM}{\mathscr{L}}

\newcommand{\Lp}[1]{\mathbf{L}_{#1}}

\newcommand{\npp}[2]{\boldsymbol{\xi}_{{#1}}^{{#2}}}

\newcommand{\distF}[2]{\boldsymbol{\delta}_{\!#2}^{#1}}
\newcommand{\da}[2]{\distF{#2}{#1}}

\newcommand{\an}[2]{\boldsymbol{\nu}_{\!#1}^{#2}}

\newcommand{\ar}[2]{\boldsymbol{r}_{\!#1}^{#2}}

\newcommand{\pp}{\mathbf{p}}

\newcommand{\qq}{\mathbf{q}}

\DeclareMathOperator{\Unp}{Unp}  

\newcommand{\mcv}[1]{ \mathbf{h}_{{#1}} }

\newcommand{\GM}{\mathscr{G}}

\newcommand{\HM}{\mathscr{H}}

\newcommand{\density}{\boldsymbol{\Theta}}

\newcommand{\unitmeasure}[1]{\boldsymbol{\alpha}(#1)}

\newcommand{\restrict}{ \mathop{ \rule[1pt]{.5pt}{6pt} \rule[1pt]{4pt}{0.5pt} }\nolimits }

\newcommand{\ud}{\ensuremath{\,\mathrm{d}}}

\newcommand{\uD}{\ensuremath{\mathrm{D}}}
\newcommand{\Der}{\uD}

\DeclareRobustCommand{\rchi}{{\mathpalette\irchi\relax}}
\newcommand{\irchi}[2]{\raisebox{\depth}{$#1\chi$}}

\newcommand{\arl}[2]{\underline{\ar{#1}{#2}}}

\DeclareMathOperator{\cato}{cato}

\newcommand{\id}[1]{\bm{1}_{#1}}

\newcommand{\lIm}{ [ }
\newcommand{\rIm}{ ] }

\newcommand{\tbcup}{{{\textstyle \bigcup}}}

\newcommand{\Clos}[1]{\mathop{\mathrm{Clos}}#1}

\newcommand{\cylind}[4]{\mathbf{C} ( #1, #2, #3, #4 )}

\newcommand{\VF}{\mathscr{X}}

\DeclareMathOperator{\trace}{trace}

\DeclareMathOperator{\Hom}{Hom}

\newcommand{\Bdry}{\partial}
\newcommand{\redbd}{\partial^{\ast}}

\DeclareMathOperator{\ap}{ap}

\DeclareMathOperator{\lin}{span}

\newcommand{\cnt}[1]{\mathscr{C}^{#1}}

\newcommand{\orthgroup}[1]{\mathbf{O}({#1})}

\DeclareMathOperator{\Int}{Int}

\DeclareMathOperator{\dmn}{dmn}

\newcommand{\grad}{\nabla}

\DeclareMathOperator{\spt}{spt}

\DeclareMathOperator{\Tan}{Tan}

\newcommand{\without}{\!\mathop{\smallsetminus}}

\DeclareMathOperator{\im}{im}

\renewcommand{\adim}{n+1}
\renewcommand{\vdim}{n}
\renewcommand{\codim}{1}

\makeatletter
\newcommand*{\nsubsection}[1]{%
  \subsection*{#1}%
  \addcontentsline{toc}{subsection}{#1}
  \NR@gettitle{#1}%
}
\makeatother

\makeatletter
\let\c@equation\c@enumi

\makeatother

\counterwithin*{equation}{theorem}
\counterwithin*{enumi}{theorem}

\title{Quadratic flatness and Regularity\\
  for Codimension-One Varifolds \\
  with Bounded Anisotropic First Variation}
\author{Sławomir Kolasiński \and Mario Santilli}

\begin{document}

\maketitle

\begin{abstract}
    We prove that if $ V $ is a $ n $-dimensional varifold in an open subset of $ \R^{n+1} $ with
    bounded anisotropic mean curvature such that $ \spt \| V \| $ has locally finite
    $ \HM^n $-measure, then $ \spt \| V \| $ can be touched by two mutually tangent balls at
    $ \HM^n $ almost all points. In particular, this result implies that $ \HM^n $ almost all of
    $ \spt \| V \| $ can be covered by the union of countably many $ \cnt{2} $-regular
    $ n $-dimensional submanifolds of $ \R^{n+1} $. Moreover, combined with Allard's local
    anisotropic regularity theorem, it implies that if $ V $ is an \emph{integral} varifold with
    bounded anisotropic mean curvature and if $ \HM^n \restrict \spt \| V \| $ is absolutely
    continuous with respect to $ \| V \| $, then $ \spt \| V \| $ is $ \cnt{1, \alpha} $-regular
    around $ \HM^n $ almost every point of density $ 1 $.
\end{abstract}

\section{Introduction}
\label{sec:intro}

In the renowned paper \cite{Allard1972}, Allard establishes several fundamental results for
varifolds whose Euclidean (isotropic) first variation is controlled. The pillar of this theory is
the \emph{monotonicity formula}; cf.~\cite[5.1]{Allard1972} or~\cite{Simon83}. If $ V $ is a
$ k $-dimensional rectifiable varifold in~$ \R^{\adim} $ with mean curvature in $ \Lp{p}(\| V \|) $
and $ p > k $, one of Allard's main achievements is contained in the celebrated regularity theorem
\cite[8.1]{Allard1972} stating that $ \spt \| V \| $ is locally a~graph of a~$ \cnt{1} $-function
around each point of density close to $ 1 $. Another classical and central consequence of the
monotonicity formula states that $ \spt \| V \| $ must be
countably $(\HM^{k},k)$-rectifiable; cf.~\cite[8.6]{Allard1972}. This conclusion has considerably
been strengthened in recent years.  For codimension-one integral varifolds $ V $ with Euclidean mean
curvature in $ \Lp{p}(\| V \|) $, and $p > \vdim$, Sch\"atzle in \cite{Schatzle04} obtained
a~quadratic-flatness result asserting that for $ \| V \| $ almost~all\ $ a $ there exist
$ \nu \in \sphere{\vdim} $ and $ r > 0 $ such that
\begin{displaymath}
        \oball{a \pm r\nu}{r} \cap \spt \| V \|
        = \varnothing \,,
\end{displaymath}
where $\oball{a}{r}$ is the open ball centred at~$a$ of radius~$r$. In particular,
quadratic-flatness implies that $ \spt \| V \| $ can be covered, up to a set of $ \HM^{\vdim} $-measure
zero, by the union of countably many $ \cnt{2} $-regular $ \vdim $-dimensional submanifolds of $ \R^{\adim} $. 
The~$ \cnt{2} $-rec\-ti\-fi\-a\-bi\-li\-ty and quadratic flatness were later extended to \emph{integral} varifolds of locally bounded first variation by Menne (cf.\ \cite{Menne13}  and \cite[Theorem 19.11]{menne2024sharplowerboundmean}). The $ \cnt{2} $-rectifiability was further generalized by the second author in \cite{Santilli21} to
arbitrary codimension varifolds  with bounded Euclidean mean curvature satisfying a~uniform lower bound on the density. We mention that $\cnt{2}$-rectifiability is sharp for varifolds
with \emph{bounded} mean curvature, while for \emph{stationary} varifolds De Lellis, Brena and
Franceschini~\cite{Brena2025} have recently proved $\cnt{\infty}$-rectifiability.

The extension of these results, in particular the extension of the regularity theory, to integral
varifolds whose first variation with respect to a~\emph{parametric elliptic integrand} is controlled
is a well~known and outstanding problem in the calculus of variations. The main obstacle arises from
the fact that no monotonicity formula is available for such varifolds; cf.~\cite[4.9]{Allard1972}
and~\cite{Allard1974}. Henceforth, most~of the powerful techniques from~\cite{Allard1972} are not
accessible in this anisotropic setting.

Significant progress was made after the groundbreaking result of De~Philippis, De~Rosa, and
Ghiraldin~\cite{Philippis2018}, where the authors introduce a~new notion of ellipticity, called
the~\emph{atomic condition} or AC for short, and prove it is necessary and sufficient to ensure that
varifolds whose anisotropic total variation measure is Radon are rectifiable -- an analogue of~\cite[5.5(1)]{Allard1972}. Later it was shown in~\cite{Rosa2020} that
this new notion is stronger than the classical Almgren's ellipticity introduced
in~\cite{Almgren1968}. The~atomic condition was also exploited by De~Rosa and
Tione~\cite{DeRosaTione2022} who introduced its uniform version and used it to show almost
everywhere regularity for Lipschitz graphs of arbitrary codimension with anisotropic mean curvature
in~$\Lp{p}(\|V\|)$ given $p > k$. In the recent work of De~Rosa, Lei, and~Young~\cite{DeRosa2025}
the connection with polyconvexity is established. In codimension-one an integrand~$F$ satisfies
the~atomic condition if and only if the associated norm~$\phi$ is strictly convex;
cf.~\cite[Theorem~1.3]{Philippis2018}.  Recently, De~Philippis, De~Rosa and Li have extended the
min-max methods to the anisotropic setting in~\cite{DePhilippis2024a} (see
also~\cite{DePhilippis2024}), hence providing the existence of hypersurfaces $ \Sigma $ in any
$(\adim)$-dimensional smooth closed Riemannian manifold which are $ F $-stationary (or, more
generally, have constant mean $ F $-curvature) with respect to a uniformly elliptic integrand $ F $,
and with singular set of vanishing $\HM^{\vdim-2}$-measure.

In this paper we study the rectifiability and the regularity properties of the support of
codimension-one varifolds $ V $ on an open subset $ \Omega $ of $ \R^{\adim} $, such that the
anisotropic mean curvature with respect to an arbitrary  uniformly convex norm~$\phi$ is bounded
in~$\Lp{\infty} $. In this setting a fundamental insight is obtained by Allard in
\cite[3.6]{Allard1986} where, employing a~barrier principle, he proved that if $ V $ satisfies the
\emph{absolute continuity hypothesis}
\begin{gather}
    \label{absolute_continuity_hp}
    \text{$ \HM^{\vdim} \restrict \spt \| V \| $ is absolutely continuous with respect to $ \| V \| $} \,,
\end{gather}
and if $ a \in \spt \| V \| $ is a unit-density point that satisfies certain smallness
  assumptions  on the $ \Lp{2} $ and
$ \Lp{\infty} $-height excess inside \emph{unbounded} cylinders in $ \Omega $ (cf.\ \cite[2.2(1)-(8), 3.6(1)-(5)]{Allard1986}), then $ a $ is a
$ \cnt{1}$-regular point of $ \spt \| V \| $ (i.e.\ $ \spt \| V \| $ is a graph of a
$ \cnt{1} $-function locally around $ a $).  This is an~extension of the well known
theorem~\cite[8.16]{Allard1972}, which holds for varifolds whose Euclidean first variation is
controlled. On the other hand, this result does not guarantee the existence of a~$ \cnt{1} $-regular
part for~$V$, as one needs to verify the required smallness assumptions at $ \| V \| $-almost all
unit-density points of $ \spt \| V \| $. For the euclidean area integrand, this issue is resolved by
a (first-order) flatness result whose proof employs again the monotonicity formula;
cf.~\cite[8.8]{Allard1972}. An analogous result in the anisotropic setting has been missing so far,
since the approach of \cite{Allard1972} clearly cannot be employed.  This has remained an open problem that we resolve in this paper. Indeed, one of our main results is a general quadratic flatness theorem for
varifolds of bounded mean $ \phi $-curvature that can be combined with Allard's theorem \cite[3.6]{Allard1986} to settle the aforementioned
regularity question. Before stating the result, we introduce some notation and recall a~few basic
facts.

If $ A \subseteq \R^{\adim} $ then we denote by $ \da{A}{} $ the distance function (taken with
respect to the Euclidean metric) from $ A $ and we define the \emph{proximal unit-normal bundle}
\begin{displaymath}
    N(A) = \bigl( \Clos{A} \times \sphere{\vdim} \bigr) \cap \bigl\{
    (a, \eta) : \da{A}{}(a + r \eta) = r \; \text{for some $ r > 0 $}
    \bigr\} \,,
\end{displaymath}
whose fiber at $ a $ is defined as $ N(A,a) = \sphere{\vdim} \cap \{\eta: (a, \eta) \in N(A)\}. $ It
is known that $ N(A) \subseteq \Clos{A} \times \sphere{\vdim} $ is countably $\vdim$~rectifiable in
the sense of~\cite[3.2.14]{Federer1969} and Borel; moreover, $ \pp \lIm N(A) \rIm $ is countably
$(\HM^{\vdim},\vdim)$~rectifiable of class $\cnt{2}$; cf.\
\ref{rem:phi_normal_bundle_in_terms_of_Euclidean_one}. Neither $ N(A) $ nor $ \pp[N(A)]$ have, in
general, locally finite $ \HM^{\vdim} $-measure (cf.\ \ref{accumulating parallel lines} and
\cite[Lemma A.3]{Santilli_Valentini_2025}).

\begin{theorem}[Quadratic flatness, \protect{cf.~\ref{main_rectifiability proof}}]
    \label{main:rectifiability}
    Suppose 
    \begin{gather}
        \text{$ \phi $ is a uniformly convex $ \cnt{2} $-norm of $ \R^{\adim} $} \,,
        \\
        \text{$ \Omega \subseteq \R^{\adim} $ is open} \,,
        \quad
        V \in \Var{\vdim}(\Omega)  \,,
        \\
        \text{$ 0 \leq  H < \infty $ is such that $ \| \delta_\phi V \| \leq H \| V \| $} \,,
        \\
        \text{$\LM^{\adim}\bigl(\spt \| V \|\bigr) =0 $
          and $ \HM^{\vdim} \restrict \pp[N(\spt \| V \|)] $
          is a Radon measure over $ \Omega $} \,,
        \label{zero volume hypothesis}
        \\
        S = \spt \| V \| \cap \big\{
        a : \exists \, \nu \in \sphere{\vdim} \, \text{s.t.} \;
        N(\spt \| V \|,a) = \{\pm \nu\}\big\} \,.
    \end{gather}
    
    Then $ \HM^{\vdim} \restrict \spt \| V \| $ is a Radon measure over $ \Omega $ and
    $ \HM^{\vdim}\bigl( \spt \| V \| \without S \bigr) =0 $. In particular, $ \spt \| V \| $ is
    $(\HM^{\vdim}, \vdim) $-rectifiable of class $ \cnt{2} $.
\end{theorem} 

As a consequence of the theorem, we notice that \eqref{zero volume hypothesis}  is \emph{equivalent} to the seemingly stronger condition \begin{equation}
 \textrm{$\HM^{\vdim} \restrict \spt \| V \| $ is a Radon measure over $ \Omega $.}
\end{equation}  
This hypothesis  is necessary to conclude that
$ \HM^{\vdim}\bigl( \spt \| V \| \without S \bigr) =0 $, as shown by the simple example in
\ref{accumulating parallel lines}.  We remark that even first order rectifiability of~$\spt \|V\|$
 is new under the hypothesis of~\ref{main:rectifiability}.

In \cite{Schatzle04} quadratic flatness is proved combining results from \cite{Allard1972} and
\cite{Brakke1978} with non-variational techniques based on maximum principles developed in the
context of viscosity solutions of second order elliptic PDE's (\cite{Trudinger89} and
\cite{Caffarelli1989}).  We refer to \cite{Almgren1986}, \cite{SolomonWhite} and \cite{Allard1986}
for the introduction of maximum principle methods in varifold's theory. Later, these methods were
systematically investigated in \cite{White2010,White2016} for the area integrand and varifolds of
arbitrary codimension, and in~\cite{DePhilippis2019} for codimension-one varifolds and elliptic
integrands arising from uniformly convex norms. More specifically, these works introduce the notion
of~$(\vdim,h)$~sets (cf.~\ref{def:n_h_sets} and~\ref{lem:support_of_a_varifold_is_nh_set}), and they
establish the connection with varifold theory by proving that the support of a varifold with bounded
mean curvature must be an $(\vdim,h)$~set (cf.\ \cite[Theorem~2.8]{White2016} and
\cite[Lemma~4.11]{DRKS2020ARMA}). See also \cite{SavinMinimalsurfaces}, \cite{Santilli20b},
\cite{DRKS2020ARMA}, \cite{Santilli21}, and \cite{DePhilippisNewproofAllard}. Since a proof of
Theorem~\ref{main:rectifiability} cannot rely on the variational machinery developed in
\cite{Allard1972} and \cite{Brakke1978} for the area integrand, we develop a completely different
approach that combines maximum-principle techniques (cf.\ \ref{lem:weak_maximum_principle}) with the
theory of curvature for arbitrary closed sets (cf.\ \cite{HugSantilli} and
\cite{Santilli20b}). Indeed, we are able to obtain the result in the general setting of
$(\vdim,h) $-sets; see Theorem \ref{lem:minkowski_content_of_nh_sets} and Remark
\ref{main_rectifiability proof}.

Employing Theorem \ref{main:rectifiability} we can now verify the hypotheses of Allard's
local regularity theorem \cite[3.6]{Allard1986} at  $ \HM^{\vdim} $ a.e.\
$ a \in \spt \| V \| $ where $ \density^{\vdim}(\| V \|, a) = 1 $. 

\begin{theorem}
    \label{main:regularity}
    Suppose 
    \begin{gather}
        \text{$ \phi $ is a uniformly convex $ \cnt{3} $-norm on $ \R^{\adim} $} \,, 0 < \alpha < 1 \,,\\
        \text{$\Omega \subseteq \R^{\adim} $ is open, $V $ is a $ \vdim $-dimensional integral varifold in $ \Omega $} \,,\\
        \text{$ 0 \leq  H < \infty $ is such that $ \| \delta_\phi V \| \leq H \| V \| $} \,,\\
        \text{$ \HM^{\vdim} \restrict \spt \| V \| $ is absolutely continuous with respect to $ \| V \| $} \,, \\ 
         Q = \spt \| V \| \cap \bigl\{ a  : \density^{\vdim}(\| V \|, a) = 1 \bigr\} \,.
    \end{gather}
    
        For $ \HM^{\vdim}$ almost all $ a \in Q $ there exists $ r > 0 $ such that
        \begin{displaymath}
            \oball ar \cap Q = \oball ar \cap \spt \| V \|
            \quad \text{is a~$ \cnt{1, \alpha} $-hypersurface} \,.
        \end{displaymath}
        \end{theorem}

Theorems \ref{main:rectifiability} and \ref{main:regularity} can be applied to study the regularity of sets of finite
perimeter, whose first variation of the anisotropic $ \phi $-perimeter is controlled.  For every set
of finite perimeter $ E \subseteq \R^{\adim}$ the $ \phi $-perimeter of $ E $ is
defined by
\begin{displaymath}
    \mathcal{P}_{\phi}(E) = \textint{\redbd E}{}\, \phi(\nu_E(x)) \ud \HM^{\vdim}(x) \,.
\end{displaymath}
Here we denote by $ \redbd E $ and $ \nu_E $ the reduced boundary and the measure-theoretic
exterior normal of $ E $; cf.~\cite{AmbrosioFuscoPallara}. If $ E \subseteq \R^{\adim} $ is a set of
finite perimeter and $ g \in \VF(\R^{\adim}) $ is a~compactly supported $ \cnt{1} $-vectorfield, then \emph{the
  first variation of $ \mathcal{P}_{\phi} $ at $ E $ in the direction $ g $} is defined by
\begin{displaymath}
    \delta \mathcal{P}_{\phi}(E)(g) = \left. \tfrac{d}{dt} \mathcal{P}_{\phi}(\varphi_t \lIm E \rIm) \right|_{t=0} \,,
\end{displaymath}
where $ \varphi : \R^{\adim}\times \R \rightarrow \R^{\adim} $ is the flow of $ g $. We denote by
$ \| \delta \mathcal{P}_{\phi}(E) \| $ the total variation measure associated with
$\delta \mathcal{P}_{\phi}(E) $. Notice that
$ \| \delta \mathcal{P}_{\phi}(E) \| = \| \delta_\phi \var{\vdim}(\redbd E) \| $, where
$ \var{\vdim}(\redbd E) $ is the unit-density varifold associated with the
rectifiable set $  \redbd E $. The following regularity result follows from \cite[Lemma 6.2]{DRKS2020ARMA} and  Theorem~\ref{main:regularity}.
 
\begin{corollary}\label{thm regularity sets of finite perimeter}
    Suppose $ 0 < \alpha < 1 $, $ 0 \leq \kappa < \infty $ and $ E \subseteq \Omega $ is a set of finite perimeter  such that  $ \HM^{\vdim}\bigl(\Clos{\redbd E} \without \redbd E\bigr) = 0$ and
$$ \| \delta \mathcal{P}_{\phi}(E) \| \leq \kappa \, \bigl( \HM^{\vdim} \restrict \redbd E \bigr). $$

   	Then there exists an open subset $ \Omega \subseteq \R^{n+1} $ and a $ \cnt{1, \alpha} $-hypersurface $ M \subseteq \redbd \Omega $ such that the following statements hold.
   \begin{enumerate}
   	\item\label{regularity theorem 1} $ \LM^{\adim}(\Omega \without E) = \LM^{\adim}(E \without \Omega) =0 $, $ \HM^n(\Bdry \Omega \setminus \redbd \Omega) =0 $, $ \spt (\HM^n \restrict \redbd \Omega) = \Bdry \Omega $.
   	\item\label{regularity theorem 2}	$ \HM^{\vdim}(\Bdry \Omega \without M) = 0$.
   	\item\label{regularity theorem 5}  There exists a countable family of $ \cnt{2} $-hypersurfaces of $ \R^{n+1}  $ that covers $ \HM^n $ almost all $ \partial \Omega $.
   	\item $ N(\partial \Omega,a) = \{\pm \an{\Omega}{}(a)\} $ for $ \HM^n $ a.e.\ $  a\in \partial^\ast \Omega $.
   \end{enumerate}
     \end{corollary}

 \begin{remark}
     \label{isoperimetric}
     The proof of~\cite[Theorem~6.5]{DRKS2020ARMA} employs almost
     everywhere regularity of sets of finite perimeter satisfying the hypothesis of
     Corollary~\ref{thm regularity sets of finite perimeter}. The authors erroneously believed this claim to be an immediate consequence of Allard's local anisotropic regularity theorem \cite[3.6]{Allard1986} (cf.\ page 1189 of
     \cite{DRKS2020ARMA}). 
\end{remark}

\section{Preliminaries}
\label{sec:prelim}

\nsubsection{Notation}
\label{sec:notation}

\begin{miniremark}
    \label{mr:basic_notation}
    
    We write $\natp$ for the set of positive integers. In~case
    $S \subseteq \R^{\adim}$ the \emph{closure} in~$\R^{\adim}$ of~$S$ is $\Clos{S}$, the
    \emph{interior} is denoted $\Int S$, and its \emph{topological boundary} in~$\R^{\adim}$
    is~$\Bdry S$. The~\emph{characteristic function} of $S \subseteq \R^{\adim}$ is
    $\CF{S} : \R^{\adim} \to \{0,1\}$. Given any set $X$ the \emph{identity map} on~$X$ is the
    function $\id{X} : X \to X$.  We~write  $\oball ar$ for the
    \emph{open ball} of radius $0 < r < \infty$ and centre~$a$ (in the metric space that~$a$ belongs
    to, which should be clear from the context).   Whenever $S \subseteq \R^{\adim}$ and $a \in \R^{\adim}$ the \emph{tangent
    	cone} $\Tan(S,a) \subseteq \R^{\adim}$ is defined as
    in~\cite[3.1.21]{Federer1969}.
    
    The \emph{scalar product} of $u,v \in \R^{\adim}$ is written as $u \bullet v$ and $|\cdot|$ is
    the associated \emph{Euclidean norm}. We define the standard \emph{polarity}
    $\beta : \R^{\adim} \to \Hom(\R^{\adim},\R)$ by $ \beta(u)v = u \bullet v $.  Note that given
    $\nu \in \sphere{\vdim}$ the hyperplane orthogonal to~$\nu$ is
    $ \ker \beta(\nu) = \lin \{ \nu \}^{\perp} $. The \emph{operator norm} with respect to
    $ | \cdot | $ is denoted by~$ \| \cdot \| $. We~set
    $\sphere{\vdim} = \R^{\adim} \cap \{ x : |x| = 1 \} $. If $ A \subseteq \R^{\adim} $ we define
    the \emph{distance function} by
    \begin{displaymath}
        \da{A}{}(x) = \inf \bigl\{ | x-a | : a \in A \bigr\}
        \qquad \text{for $ x \in \R^{\adim} $} \,.
    \end{displaymath}
    and the \emph{equidistant set} by
    $ S(A,r) = \R^{\adim} \cap \bigl\{ x : \da{A}{}(x) = r \bigr\} $ for $ r > 0 $.
    
    For $k,m \in \natp$ with $0 \le k \le m$ we~denote by $ \grass mk $ the \emph{Grassmannian} of
    $ k $-dimensional planes in $ \R^m $. If $ T \in \grass mk $ then we use the same symbol $ T $
    to denote the orthogonal projection $ T : \R^m \rightarrow \R^m $ onto $ T $. If
    $T \in \grass mk$, we set $T^{\perp} = \id{\R^{m}} - T \in \grass{m}{m-k} $, the orthogonal
    complement of $ T $.
\end{miniremark}

\begin{miniremark}
    As in~\cite[p.~669]{Federer1969} given a~relation $ Q \subseteq \R^{\adim} \times \R^{\adim}$
    and $ S \subseteq \R^{\adim}$ the \emph{restriction of $Q$ to $S$} is
    \begin{displaymath}
        Q | S = Q \cap \bigl\{ (x,y) : x \in S \bigr\} \,.
    \end{displaymath}
\end{miniremark}

\begin{miniremark}\label{projections}
    The \emph{canonical projections} for the product
    $\R^{\adim} \times \R^{\adim}$ are given by
    \begin{gather}
        \pp  : \R^{\adim} \times \R^{\adim} \rightarrow \R^{\adim}
        \quad \text{and} \quad
        \qq  : \R^{\adim} \times \R^{\adim} \rightarrow \R^{\adim}
        \\
        \text{so that} \quad
        \pp(a, \eta) = a
        \quad \text{and} \quad
        \qq(a, \eta) = \eta
        \quad \text{for $a,\eta \in \R^{\adim}$} \,.
    \end{gather}
\end{miniremark}

\begin{miniremark}
    \label{def:cylinder}
    Given a linear projection $T : \R^{m} \to \R^{m}$ and $r,s \in \{ t : 0 < t \le \infty \} $
    we~define the \emph{open cylinder}
    \begin{displaymath}
        \cylind Tars =  \R^{m} \cap \bigl\{
        z : |(\id{\R^{m}} - T)(z-a)| < s ,\; |T(z-a)| < r
        \bigr\}  \,.
    \end{displaymath}
\end{miniremark}

\begin{miniremark}
    If $\nu \in \sphere{\vdim}$, $A \subseteq \R^{\adim}$, and $ f : A \to \R $ we set
    \begin{gather}
        \graph_{\nu}(f) = \R^{\adim} \cap \bigl\{ \sigma + f(\sigma) \nu : \sigma \in A  \bigr\} 
        \\
        \text{and} \quad
        \cato_{\nu}(f) = \R^{\adim} \cap  \bigl\{ \sigma + t \nu : \sigma \in A ,\, t < f(\sigma) \bigr\} \,.
    \end{gather}
    Clearly the definitions of $\graph_{\nu}(f)$ and $\cato_{\nu}(f)$ make most sense
    if $\nu \notin \lin A$.
\end{miniremark}

\begin{miniremark}
    In the whole paper we assume $\phi$~is a~uniformly convex~$\cnt{2}$-norm
    on~$ \R^{\adim}$. Uniform convexity implies that there exists
    an~\emph{ellipticity constant} $ \gamma(\phi) > 0 $ such that
    \begin{displaymath}
        \Der^2 \phi(u)(v,v) \ge \gamma(\phi) |v|^2
        \quad \text{for $ u \in \sphere{\vdim}$ and $ v \in \lin \{ u \}^{\perp} $} \,.
    \end{displaymath}
    For $l \in \nat$ we also define 
    \begin{gather}
        c_{l}(\phi) = \sup \big\{ \|\Der^k \phi(\nu) \| : \nu \in \sphere{\vdim} ,\; k \in \integers ,\; 0 \le k \le l \big\} \,.
    \end{gather}
    Note that since $\phi$ is $\cnt{2}$-regular away from the origin the constant~$\gamma(\phi)$
    coincides with the constant named ``$\gamma$'' in~\cite[3.1(4)]{Allard1986}.
\end{miniremark}

\begin{miniremark}[\protect{cf.~\cite[3.2.14 and 3.2.29]{Federer1969}}]
    \label{def:rectifiability_of_higher_class}
    Let $k,l \in \natp$ and $k \le \adim$. A set $\Sigma \subseteq \R^{\adim}$ is called
    \emph{countably $(\HM^k,k)$~rectifiable of class $\cnt{l}$} if there exists a~countable family
    $\mathcal{A}$ of~submanifolds of~$\R^{\adim}$ of~dimension~$k$ and class~$\cnt{l}$ such that
    $\HM^k(\Sigma \without \tbcup \mathcal{A}) = 0$. If, additionally, $\HM^{k}(\Sigma) < \infty$,
    we say that $\Sigma$ is~\emph{$(\HM^k,k)$~rectifiable of class $\cnt{l}$}. In~case $l=1$, we
    omit ``of class $\cnt{l}$''.
\end{miniremark}

\subsection*{Anisotropic first variation and barrier principle}  
\addcontentsline{toc}{subsection}{Anisotropic first variation and barrier principle}

\begin{miniremark}
	A $ \vdim $-dimensional varifold $ V $ in an open subset $ U $ of $ \R^{\adim} $ is a Radon
	measure over $ U \times \grass{\adim}{\vdim} $, which is often identified with a~positive linear
	functional on $\ccspace{U \times \grass{\adim}{\vdim}}$ -- the space of continuous compactly
	supported real-valued functions on $U \times \grass{\adim}{\codim}$; cf.~\cite[2.5.14,
	2.5.19]{Federer1969}.  The~space of all $\vdim$-dimensional varifolds in~$\Omega$ is denoted
	$\Var{n}(\Omega)$. A $ \vdim $-dimensional varifold $ V $ on $ U $ is called \emph{integral} if
	there exist a countably $ (\HM^{\vdim}, \vdim) $-rectifiable set $ M \subseteq U $ and
	a~$\natp$-valued locally $ \HM^{\vdim} \restrict M $-summable function $ \theta $ such that
	\begin{displaymath}
		V(\alpha) =  \textint{M}{} \theta(x)\, \alpha(x, \Tan^{\vdim}(\HM^{\vdim} \restrict M, x))\ud\HM^{\vdim}(x)
	\end{displaymath}
	for $ \alpha \in \ccspace{U \times \grass{\adim}{\vdim}} $. In case $\theta(x) = 1$ for
	$\HM^{\vdim}$~almost all $x \in M$ we write $\var{\vdim}(M)$ for the corresponding varifold. The
	space of all $\vdim$-dimensional integral varifolds in~$\Omega$ is denoted by
	$\IVar{\vdim}(\Omega)$
	
	Suppose $ \phi $ is a $ C^1 $-norm on $ \R^{\adim} $.  For $ \nu \in \sphere{\vdim} $, we define
	$ B_\phi(\nu) \in {\rm Hom}\bigl(\R^{\adim},\R^{\adim}\bigr) $ by the formula
	\begin{displaymath}
		\langle v, B_\phi(\nu) \rangle
		= \phi(\nu) v - \langle v, \Der \phi(\nu) \rangle \, \nu
		\quad \textrm{for $ v \in \R^{\adim} $.}
	\end{displaymath}  
	Notice that $ B_\phi(-\nu) = B_\phi(\nu) $ for each $ u \in \sphere{\vdim} $. For
	$V \in \Var{\vdim}(\Omega)$ we define the $ \phi $-weighted varifold
	$V_\phi \in \Var{k}(\Omega)$ by
	\begin{displaymath}
		V_\phi(\alpha) = \textint{}{} \alpha(x,S) \phi(\nu(S)) \ud V(x,S)
		\quad \text{for $\alpha \in \ccspace{\Omega \times \grass{\adim}{k}}$} \,,
	\end{displaymath}
	and the first $ \phi $-variation of $ V $ in $ U $ by
	\begin{displaymath}
		\delta_\phi V (g) =  \textint{}{} \Der g(x) \bullet B_\phi(\nu(S)) \ud V(x, S)
		\quad \textrm{for $ g \in \VF(U) $,}
	\end{displaymath}
	where $ \nu : \grass{\adim}{\vdim} \rightarrow \sphere{\vdim} $ is a Borel map such that
	$ \nu(S) \in S^\perp \cap \sphere{\vdim} $ for each $ S \in \grass{\adim}{\vdim} $. Notice that
	these definitions do not depend on the choice of~$\nu $.  The total variation
	$ \| \delta_\phi V \| $ of $ \delta_\phi V $ is the largest Borel regular measure over $ U $
	such that
	\begin{displaymath}
		\| \delta_\phi V \|(W) = \sup \bigl\{ \delta_\phi V(g) : g \in \VF(W), \; | g | \leq 1 \bigr\}
	\end{displaymath}
	for each open subset $ W \subseteq U $. If $ \| \delta_\phi V \| $ is a Radon measure over
	$ U $, then by the standard Lebesgue differentiation theory implies the existence of
	a~$ \R^{\adim} $-valued $ \| V \| $-integrable function $ \mcv{\phi}(V, \cdot) $ and of a
	$ \sphere{\vdim} $-valued $ \| \delta_\phi V \| $-integrable function
	$ \boldsymbol{\eta}_\phi(V, \cdot) $ such that
	\begin{displaymath}
		\delta_\phi V (g)
		= \textint{}{} \mcv{\phi}(V, x) \bullet g(x)\ud\| V_\phi \|(x)
		+ \textint{}{} \boldsymbol{\eta}_\phi(V, x) \bullet g(x)\ud\| \delta_\phi V \|_{\mathrm{sing}}(x)
	\end{displaymath}
	for every $ g \in \VF(U) $, where $ \| \delta_\phi V \|_{\mathrm{sing}} $ is the singular part
	of $ \| \delta_\phi V \| $ with respect to~$ \| V \| $.  Notice that if $ V $ is a
	$ \vdim $-dimensional varifold of $ U $ such that there exists $ \kappa \geq 0 $ for which
	\begin{displaymath}
		\| \delta_\phi V \| \leq \kappa \, \| V \|,
	\end{displaymath}
	then $ \| \delta_\phi V \| $ is a Radon measure over $ U $,
	$ \| \delta_\phi V \|_{\rm sing} =0 $ and
	$ \mcv{\phi}(V,\cdot) \in \Lp{\infty}(\| V \|, \R^{\adim}) $.
\end{miniremark}
\begin{miniremark}
    \label{mr:mcv_of_smooth_hypersurfaces}
    Suppose $ N \subseteq \Omega $ is open and such that $ \Omega \cap \Bdry N $
    is a~smooth hypersurface. Let
    $V = \var{\vdim}(\Bdry N) \in \Var{\vdim}(\Omega)$ and
    $ \nu_{N} : \Omega \cap \Bdry N \rightarrow \sphere{\vdim} $ be the
    \emph{inward pointing} unit-normal. Referring
    to~\cite[Proposition~1]{DePhilippis2019} or~\cite[Remark~2.21]{DRKS2020ARMA}
    we get a~formula for the first variation of~$V$ with respect to $\phi$
    \begin{displaymath}
        \delta_{\phi} V(g) = - \int_{\Omega \cap \Bdry N}
        \phi(\nu_{N}(x)) \mcv{\phi}(V,x) \bullet g(x)
        \ud \HM^{\vdim}(x)
        \quad \text{whenever $ g \in \VF(\Omega) $} \,,
    \end{displaymath}
    where
    \begin{displaymath}
        \mcv{\phi}(V,x)
        =  \mcv{\phi}(\Bdry N,x)
        = \frac{-\trace\big(\Der(\grad \phi \circ \nu_{N})(x)\bigr)}{\phi(\nu_{N}(x))} \nu_{N}(x)
        \quad
        \text{for $ x \in \Omega \cap \Bdry N $} \,.
    \end{displaymath}
\end{miniremark}

\nsubsection{A barrier principle}  

\begin{definition}[\protect{cf.\ \cite[Definition 3.1]{DePhilippis2019}}]
    \label{def:n_h_sets}
    Suppose $ \Omega \subseteq \R^{\adim} $ is open, $ 0 \le h < \infty $ and $ A \subseteq \Omega $
    is relatively closed. We say that $ A $ is an \emph{$ (\vdim,h) $~subset of $ \Omega $ with
      respect to $ \phi $} provided the following property holds: if $ N \subseteq \Omega $ is an
    open set such $ \Omega \cap \Bdry N $ is a~smooth hypersurface,
    $ \nu_N : \Omega \cap \Bdry N \rightarrow \sphere{\vdim} $ is the inward pointing unit-normal
    and $ A \subseteq \Clos{N} $ then
    \begin{displaymath}
        \phi(\nu_{N}(x))\mcv{\phi}(\Bdry N,x) \bullet \nu_{N}(x) \le h
        \qquad \textrm{for $ x \in A \cap \Bdry N $} \,.
    \end{displaymath}
\end{definition}

\begin{lemma}[\protect{cf.~\cite[Lemma~4.11]{DRKS2020ARMA}}]
    \label{lem:support_of_a_varifold_is_nh_set}
    Suppose $\Omega \subseteq \R^{\adim} $ is open, $V \in \Var{\vdim}(\Omega)$,
    $0 < h < \infty$, and $\|\delta_{\phi} V\| \le h
    \|V_\phi\|$. Then $ \spt \| V \| $ is an $(\vdim,h)$~subset of
    $\Omega$ with respect to $ \phi $.
\end{lemma}

We need to apply the \emph{barrier principle} with non smooth barriers.
The~following lemma is enough for our purposes.

\begin{lemma}[\protect{cf.\ \cite[Lemma~4.8]{DRKS2020ARMA}}]
    \label{lem:weak_maximum_principle}
    Suppose
    \begin{gather}
        \nu \in \sphere{\vdim}  \,,
        \quad
        T =\ker \beta(\nu) \in \grass{\adim}{\vdim} \,,
        \quad
        0 \in A \subseteq \R^{\adim} \,,
        \\
        g :  T \rightarrow \R  \quad \text{is twice pointwise differentiable at $ 0 $} \,,
        \quad
        g(0) = 0 \,,
        \quad
        \Der g(0) = 0 \,,
        \\
        \tau_1, \ldots, \tau_{\vdim} \quad \text{ is an orthonormal basis of $  T $} \,,
        \quad
        \rchi_1 \le \ldots \le \rchi_{\vdim} \in \R \,,
        \\
        \Der ^2 g(0)(\tau_i, \tau_j)
        = \rchi_i (\tau_i \bullet \tau_j)
        \quad \text{for $ i, j \in \{1, \ldots , \vdim\} $} \,,
        \quad
        \varepsilon, \delta \in \{ t : 0 < t < \infty\} \,,
        \\
        A \cap \cylind{T}{0}{\delta}{\varepsilon}
        \subseteq \Clos{(\cato_{\nu}(g))}  \cap \cylind{T}{0}{\delta}{\varepsilon} \,,
        \\
        \text{and} \quad
        A \cap \cylind{T}{0}{\delta}{\varepsilon}
        \quad \text{is an $ (\vdim,h) $~subset of $ \cylind{T}{0}{\delta}{\varepsilon} $
          with respect to $ \phi $} \,.
    \end{gather}
    Then
    \begin{displaymath}
        \textsum{i=1}{\vdim} \rchi_i\,\Der^2\phi(\nu)(\tau_i, \tau_i)   \le h.
    \end{displaymath} 
\end{lemma}

\begin{proof}
    Let $ \lambda > 0 $. We define
    $\psi_{\lambda}(\sigma) = \frac 12 \big(\Der^{2}g(0)(\sigma,\sigma)\bullet \nu + \lambda
    |\sigma|^{2}\bigr)$ for $ \sigma \in T$ and $ N_\lambda = \cato_{\nu}(\psi_{\lambda}) $. We
    notice that $ \grad \psi_{\lambda}(0) =0 $ and
    \begin{displaymath}
        \Der^2 \psi_{\lambda}(0)(\tau_i, \tau_j) = (\rchi_i + \lambda)(\tau_i \bullet \tau_j)
        \quad \textrm{for every $ i, j \in \{1, \ldots, \vdim\} $.}
    \end{displaymath}
    Choosing $ \delta $ smaller if necessary, we have that
    \begin{displaymath}
        \cato_{\nu}(g) \cap \cylind{T}{0}{\delta}{\varepsilon}
        \subseteq N_\lambda \cap \cylind{T}{0}{\delta}{\varepsilon} \,.
    \end{displaymath}
    Noting that
    \begin{displaymath}
        \nu_{N_\lambda}(\sigma + \psi_{\lambda}(\sigma) \nu)
        = \frac{\grad \psi_{\lambda}(\sigma) -\nu}{\sqrt{1 + |\grad \psi_{\lambda}(\sigma)|^2}}
        \quad \text{for  $ \sigma \in T $}
    \end{displaymath}
    and, consequently,
    \begin{displaymath}
        \Der \nu_{N_\lambda}(0) v \bullet w = \Der^2\psi_{\lambda}(0)(v, w)
        \quad \text{for $ v, w \in  T $} \,,
    \end{displaymath}
    we deduce, using $\uD(\grad \psi_{\lambda})(0)u \bullet v = \uD^2\psi_{\lambda}(0)(u,v)$ for
    $u,v \in T$, $\nu_{N_{\lambda}}(0) = -\nu$, $\uD^2 \phi(-\nu) = \uD^2 \phi(\nu)$, and recalling
    \ref{mr:mcv_of_smooth_hypersurfaces} and \ref{def:n_h_sets}, that
    \begin{displaymath}
        \textsum{i=1}{\vdim} (\rchi_i + \lambda)\Der^2\phi(\nu)(\tau_i, \tau_i)
        =\trace\big(\Der(\grad \phi)(\nu) \circ \Der \nu_{N_\lambda}(0)\bigr) \le h \,.
    \end{displaymath}
    Passing to the limit~$\lambda \to 0^{+}$ yields the claim.
\end{proof}

\nsubsection{Normal bundle and principle curvatures}

In this section we review basic definitions and results about the unit-normal bundle and the
principal curvatures of arbitrary closed sets. We refer to \cite{Kolasinski2023b} and
\cite{HugSantilli}, where these notions are given in a general $ \phi $-anisotropic setting, where
$ \phi $ is a uniformly convex $ \cnt{2} $-norm. However, for the purpose of the present paper we
only need $ \phi = | \cdot | $ to be the Euclidean norm (induced by a scalar product $ \bullet
$). Hence, in relation with the notation used in the aforementioned references, we omit the
dependencies of $ \phi $ in all the subsequent definitions.

\begin{definition}\label{def unit normal bundle}
    Suppose $A \subseteq \R^{\adim}$.
    \begin{enumerate}
\item The \emph{unit-normal bundle} is
        \begin{displaymath}
            N(A) = \bigl( \Clos{A} \times \sphere{\vdim} \bigr)
            \cap \bigl\{
            (a, \eta)
            : \da{A}{}(a + r \eta) = r \; \text{for some $ r > 0 $}
            \bigr\} 
        \end{displaymath}
        and its fibre over $a \in \Clos A$ is denoted $   N(A,a)
        = \bigl\{ \eta : (a,\eta) \in N(A) \bigr\} $.
    \item The \emph{reach function} is
        \begin{displaymath}
            \ar{A}{}(a, \eta) = \sup \R \cap \bigl\{ r : \da{A}{}(a + r \eta) = r \bigr\} \in (0, +\infty]
            \qquad \text{for  $(a, \eta) \in N(A) $} \,.
        \end{displaymath}
    \end{enumerate}
\end{definition}

\begin{remark}
    \label{rem:phi_normal_bundle_in_terms_of_Euclidean_one}
    The function $\ar{A}{}$ is upper semi-continuous,
    by~\cite[Lemma~2.35]{Kolasinski2023b}. Moreover,
    $ N(A) \subseteq \Clos{A} \times \sphere{\vdim} $ is countably $\vdim$~rectifiable in the sense
    of~\cite[3.2.14]{Federer1969} and Borel; see e.g.~\cite[Remark~4.3]{Santilli20a}. Finally,
    employing the structural theorem on the singularities of closed sets in \cite{Menne2019a}, we
    deduce that $ \pp \lIm N(A) \rIm $ is countably $(\HM^{\vdim},\vdim)$~rectifiable of class $\cnt{2}$.
\end{remark}

\begin{definition}
    \label{def:unp_npp}
    Suppose $ A \subseteq \R^{\adim} $.
    \begin{enumerate}
    \item The \emph{nearest point projection} is the multivalued map
        \begin{displaymath}
            \npp{A}{} (x) = \Clos{A} \cap \bigl\{ a : \da{A}{}(x) = |x-a| \bigr\}
            \qquad \textrm{for $ x \in \R^{\adim} $} \,.
        \end{displaymath}
    \item  We define
        \begin{displaymath}
            \an{A}{}(x) = \frac{x - \npp{A}{}(x)}{\da{A}{}(x)}
            \qquad \textrm{for $ x \in \R^{\adim} \without \Clos{A} $} \,.
        \end{displaymath}
    \item The set where $ \npp{A}{}(x) $ is a singleton (i.e. there exists
        a~\emph{unique nearest point}) is denoted
        \begin{displaymath}
            \Unp(A) = \R^{\adim} \cap \bigl\{ x : \HM^0(\npp{A}{}(x)) = 1 \bigr\} \,.
        \end{displaymath}
    \end{enumerate}
\end{definition}

\begin{remark}
    We shall treat $\npp{A}{}|\Unp(A)$ as a single-valued function and write $\npp{A}{}(x) = a$ whenever
    $x \in \Unp(A)$ and $a \in A$ satisfy $\npp{A}{}(x) = \{a\}$.
\end{remark}

\begin{remark}
    \label{rem:basic_props_of_npp_and_dist}
    It is known that the set $\Unp^{}(A) \without \Clos{A}$ \emph{equals} the set of points in
    $ \R^{\adim} \without \Clos{A} $, where $ \da{A}{} $ is differentiable; hence, Borel;
    cf.~\cite[Lemma~2.41(c)]{Kolasinski2023b} and~\cite[4.6(1)]{Menne2019}. Therefore, by the
    Rademacher theorem, $ \npp{A}{}(x) $ is a singleton for~$ \LM^{\adim} $ almost~all
    $ x \in \R^{\adim} \without \Clos{A} $.
\end{remark}

    We now proceed to differentiate the multivalued map $ \an{A}{} $ (and,
    analogously,~$ \npp{A}{} $). We~say that $ \an{A}{} $ is
    \emph{differentiable at $ x \in \R^{\adim} \without \Clos{A} $} if
    $ \an{A}{}(x) $ is a singleton and there exists a linear map
    $ L : \R^{\adim}\rightarrow \R^{\adim} $ such for any $ \varepsilon > 0 $
    there exists $ \delta > 0 $ such that
    \begin{displaymath}
        | w - \an{A}{}(x) - L(y-x)|\le \varepsilon |y-x|
        \quad \text{for $ y \in \oball{x}{\delta}$ and $ w \in \an{A}{}(y)$} \,.
    \end{displaymath}
    Evidently, if $ x \in \R^{\adim}\without \Clos{A} $, then $ \an{A}{} $
    is differentiable at $ x $ if and only if $ \npp{A}{} $ is
    differentiable at $ x $. The linear map $ L $ is unique and it is denoted by
    $ \Der \an{A}{}(x) $. The set of differentiability points of
    $ \an{A}{} $ is the domain $ \dmn \Der \an{A}{} $ of the map
    $ \Der \an{A}{} $. The following result summarizes some of the main facts on
    $ \Der \an{A}{} $.

\begin{lemma}[\protect{cf.\ \cite[Theorem~1.5, Lemma~4.12, Lemma
      4.15]{Kolasinski2023b}\footnote{In~\cite{Kolasinski2023b}\label{lem:diff_points_normal_map}
        we denote with $ \Sigma^{}_2(A) $ the set of points in
        $ \R^{\adim}\without A $ where the distance function $ \da{A}{} $ is
        not pointwise twice differentiable. This set coincides with
        $ \R^{\adim} \without (A \cup \dmn \Der \an{A}{}) $ by
        \cite[2.41(e)]{Kolasinski2023b}.}}]
    \label{theo:distance_twice_diff}
    Assume $ \varnothing\neq A \subseteq \R^{\adim}$. Then
    \begin{enumerate}
    \item $\LM^{\adim} \bigl( \R^{\adim} \without (A \cup \dmn \Der \an{A}{}) \bigr) = 0$;
    \item
        \label{theo:distance_twice_diff:almost_all_fibers_in_dmn_Dnu}
        $\bigl\{ a + r \eta: 0 < r < \ar{A}{}(a, \eta) \bigr\} \subseteq
        \dmn \Der \an{A}{}$ for $\HM^{\vdim}$ almost all
        $(a, \eta) \in N(A)$;
    \item there exists a Borel measurable function
        $ \arl{A}{}: N(A) \rightarrow [0, +\infty]$ such that
        \begin{gather}
            \arl{A}{}(a, \eta) = \ar{A}{}(a, \eta) \quad \text{for $ \HM^{\vdim} $ a.a. $(a, \eta) \in N(A) $}
            \\
            \begin{gathered}
                \text{and if there exists $0 < s < \arl{A}{}(a, \eta)$
                  with $a + s \eta \in \dmn \Der \an{A}{}$,}
                \\
                \text{then $a + r \eta \in \dmn \Der \an{A}{}$
                  for all $0 < r < \arl{A}{}(a, \eta)$} \,;
            \end{gathered}
        \end{gather}
    \item
        \label{theo:distance_twice_diff:diagnoalizable}
        if $ x \in \dmn \Der \an{A}{} $ and $T= \Tan(S(A,r),x) $, then
        $ \Der \an{A}{}(x)(\an{A}{}(x)) =0 $, $ \Der \an{A}{}(x) $ maps $ T $ into $ T $ and
        $ \Der \an{A}{}(x)|T $ is symmetric with all the eigenvalues smaller or equal that
        $ \da{A}{}(x)^{-1} $ (in particular $ \Der \an{A}{}(x) $ is symmetric).
    \end{enumerate}
\end{lemma}

\begin{definition}[\protect{cf.~\cite[Definition~3.6]{HugSantilli}}]
    Let $A \subseteq \R^{\adim}$. We set
    \begin{displaymath}
        \widetilde{N}(A)
        = N(A) \cap \bigl\{
        (a,\eta) 
        : a  + r \eta  \in \dmn \Der \an{A}{} \text{ for some } 0 < r < \arl{A}{}(a, \eta)
        \bigr\} \,.
    \end{displaymath}
\end{definition}

\begin{remark}
    \label{rem:def:kappas}
     Note that $\HM^{\vdim}(N(A) \without \widetilde{N}(A)) = 0$
    by~\ref{theo:distance_twice_diff}\ref{theo:distance_twice_diff:almost_all_fibers_in_dmn_Dnu}.
    Suppose $ (a, \eta) \in \widetilde{N}(A) $ and $ 0 < r < \arl{A}{}(a, \eta)$. Referring
    to~\ref{theo:distance_twice_diff}\ref{theo:distance_twice_diff:diagnoalizable}
    and~\cite[Lemma~4.12(c)]{Kolasinski2023b} we let
    \begin{displaymath}
        \rchi_{A,1}(a+r\eta) \le \ldots \le \rchi_{A,\vdim}(a+r\eta)
    \end{displaymath}
    be the eigenvalues of $\Der \an{A}{}(a + r\eta)|\Tan(S(A,r), a + r \eta)$. As proved
    in~\cite[Lemma~3.5, see also Remark~3.7]{HugSantilli}, the number
    \begin{displaymath}
        \kappa_{A,i}(a, \eta)
        = \frac{\rchi_{A,i}(a + r \eta)}{1-r\rchi_{A,i}(a + r \eta)}
        \in (-\infty,\infty]
    \end{displaymath}
    does not depend on $ r $ and the functions
    $ \kappa_{A,i}: \widetilde{N}(A) \rightarrow (-\infty,\infty] $ are Borel measurable (recall
    that $ 1-r\rchi_{A,i}(a + r \eta) \geq 0 $ by Lemma~\ref{theo:distance_twice_diff}).
\end{remark}

\begin{definition}[\protect{cf.~\cite[Definition~3.11]{HugSantilli}}]
    \label{def:normal_bundle_stratification}
    Let $A \subseteq \R^{\adim}$ and
    $ \kappa_{A,1} \le \ldots \le \kappa_{A,\vdim} $ be as
    in~\ref{rem:def:kappas}. We define for $d \in \{ 1, \ldots , \vdim-1\}$
    \begin{gather}
        \widetilde{N}_d(A)
        = \widetilde{N}(A) \cap \bigl\{ (a,\eta)
        : \kappa_{A, d}(a,\eta) < \infty ,\,
        \kappa_{A, d+1}(a,\eta) = \infty
        \bigr\}
        \,,
        \\
        \widetilde{N}_0(A)
        = \widetilde{N}(A) \cap \bigl\{
        (a,\eta) : \kappa_{A, 1}(a,\eta) = \infty
        \bigr\} \,,
        \\
        \text{and} \quad
        \widetilde{N}_{\vdim}(A)
        =  \widetilde{N}(A) \cap \bigl\{
        (a,\eta)  : \kappa_{A,\vdim}(a,\eta) < \infty
        \bigr\} \,.
    \end{gather}
\end{definition}

\begin{definition}
    Let $A \subseteq \R^{\adim}$. A~\emph{principal frame-field of~$A$}
    is a Borel map
    $ \tau = (\tau_1,\ldots,\tau_{\vdim}) : \widetilde{N}^{}(A) \to
    (\sphere{\vdim})^{\vdim} $, such that for
    $(a, \eta) \in \widetilde{N}^{}(A)$, $ 0 < r < \arl{A}{}(a, \eta) $, and
    $ i\in \{ 1, \ldots , \vdim \} $ there holds
    \begin{gather}
        \tau_1(a,\eta), \ldots , \tau_{\vdim}(a, \eta)
        \quad \text{is an orthonormal basis of $\Tan(S(A,r), a + r \eta)$} \,,
        \\
        \text{and} \quad
        \Der \an{A}{}(a + r\eta)\tau_i(a, \eta)
        = \rchi_{A,i}^{}(a + r\eta) \tau_i(a, \eta) \,.
    \end{gather}
\end{definition}

\begin{lemma}
    \label{lem:pff_exists}
    Let $A \subseteq \R^{\adim}$ be closed. Then there exists a~principal frame field of~$A$.
\end{lemma}

\begin{proof}
    Let $e_1, \ldots, e_{\adim}$ be the standard basis of $\R^{\adim}$, let $ X $ be the vectorspace
    of symmetric endomorphisms of $ \R^{\adim} $ and $ \orthgroup{\adim} $ the orthogonal group of
    $ \R^{\adim} $. Define
    \begin{displaymath}
        \Lambda
        = \R^{\adim} \cap \bigl\{
        (\lambda_1, \ldots, \lambda_{\adim})
        : \lambda_1 \le \lambda_2 \le \cdots \le \lambda_{\adim}
        \bigr\}
    \end{displaymath}
    and the multifunction $ G $ (cf.\ \cite{CastaingValadier}) from $ X $ to closed subsets of
    $ \R^{\adim} \times \orthgroup{\adim} $,
    \begin{displaymath}
        G(L) = (\Lambda \times \orthgroup{\adim}) \cap \{(\lambda, u) : L(u(e_i))
        = (\lambda \bullet e_i) u(e_i) \; \textrm{for $ i = 1, \ldots, \adim$}\} 
    \end{displaymath}
    for $L \in \End{\R^{\adim}}$. Noting that
    $ \graph(G) = \{(L, \lambda, u) : L \in X, \; (\lambda, u) \in G(L)\} $ is a closed subset of
    $ X \times \R^{\adim} \times \orthgroup{\adim} $, we can apply Aumann selection theorem
    \cite[Theorem III.22]{CastaingValadier} to ensure the existence of a Borel map
    $ g : X \rightarrow \R^{\adim} \times \orthgroup{\adim} $ such that $ g(L) \in G(L) $ for each
    $ L \in X $.
    
    Recalling~\ref{rem:def:kappas}, \cite[\S{2.5} and Lemma~2.3]{HugSantilli},
    and~\ref{theo:distance_twice_diff} we see that $\widetilde{N}^{}(A)$ and $\dmn \uD \npp{A}{}$
    are Borel sets, $\uD \npp{A}{} : \dmn \uD \npp{A}{} \to \End{\R^{\adim}}$ and $\arl{A}{}$ are
    Borel maps; hence, we conclude that $\uD \an{A}{}$ is also a Borel map. Define
    \begin{displaymath}
        \Upsilon : \widetilde{N}^{}(A) \to \End{\R^{\adim}}
        \quad \text{by} \quad
        \Upsilon(a,\nu) = \uD \an{A}{}( a + \tfrac 12 \arl{A}{}(a,\nu) \nu )
    \end{displaymath}
    for $(a,\nu) \in \widetilde{N}^{}(A)$. Clearly $\Upsilon$ is Borel as a composition of Borel
    maps.  From~\ref{theo:distance_twice_diff}\ref{theo:distance_twice_diff:diagnoalizable} we
    conclude that $\im \Upsilon \subseteq X$. Denoting by
    $\sigma : \R^{\adim} \times \orthgroup{\adim} \to \orthgroup{\adim}$ the projection onto the
    second factor, we notice that $ \mu = \sigma \circ g \circ \Upsilon $ is a Borel map from
    $\widetilde{N}(A)$ into $ \orthgroup{\adim} $.  Since
    $ g(\Upsilon(a, \nu)) \in G(\Upsilon(a, \nu)) $ for each $(a, \nu) \in \widetilde{N}(A) $, there
    exists for each $ (a, \nu) \in \widetilde{N}(A) $ a point $ \lambda \in \Lambda $ such that
    \begin{displaymath}
        \langle \mu(a, \nu)(e_i), \Upsilon(a, \nu) \rangle 
        = (\lambda \bullet e_i)\, \mu(a, \nu)(e_i) 
        \quad \textrm{for $ i = 1, \ldots , \adim $} \,. 
    \end{displaymath}
    In particular, $ \{\mu(a, \nu)(e_i): i = 1, \ldots, \adim\} $ is an orthonormal basis of
    $ \R^{\adim} $ of eigenvectors of $ \uD \an{A}{}( a + \tfrac 12 \arl{A}{}(a,\nu) \nu ) $, for
    each $ (a, \nu) \in \widetilde{N}(A) $. For $ i = 1, \ldots, \adim $ we define
    \begin{displaymath}
        B_i = \widetilde{N}(A) \cap \{ (a, \nu) : \mu(a, \nu)(e_i) = \nu \}
    \end{displaymath}
    and we notice from~\ref{theo:distance_twice_diff}\ref{theo:distance_twice_diff:diagnoalizable}
    that the sets $ B_i $ form a Borel partition of $ \widetilde{N}(A) $. Finally, for
    $ i = 1,\ldots, \adim $ and $ (a, \nu) \in B_i $ we define
    \begin{gather}
      \tau_1(a, \nu) = \mu(a, \nu)(e_1), \ldots, \tau_{i-1}(a, \nu) = \mu(a, \nu)(e_{i-1}), 
      \\
      \tau_{i}(a, \nu) = \mu(a, \nu)(e_{i+1}), \ldots \tau_{\vdim}(a, \nu) = \mu(a, \nu)(e_{\adim}) \,;
    \end{gather}
    hence, the resulting maps
    $ \tau_1, \ldots, \tau_{\vdim} : \widetilde{N}(A) \rightarrow \sphere{\vdim} $ form a principal frame
    field of~$ A $.
\end{proof}

\begin{lemma}[\protect{\cite[Lemma~3.9]{HugSantilli}}] 
    \label{lem:HS39}
    Let $A \subseteq \R^{\adim}$ and $\tau = (\tau_1,\ldots,\tau_{\vdim})$ be
    a~principal frame-field of~$A$. For $i \in \{1,2,\ldots,\vdim\}$
    define $\zeta_i : \widetilde{N}^{}(A) \to \R^{\adim} \times \R^{\adim}$ by
    \begin{displaymath}
        \zeta_i(a,\eta) = \left\{
            \begin{aligned}
              &\pp^{*}(\tau_i(a,\eta)) + \kappa_{A,i}(a,\eta) \qq^{*}(\tau_i(a,\eta)) 
              &&\text{if $\kappa_{A,i}(a,\eta) < \infty$} \,,
              \\
              &\qq^{*}(\tau_i(a,\eta))
              &&\text{if $\kappa_{A,i}(a,\eta) = \infty$} \,,
            \end{aligned}
        \right.
    \end{displaymath}
    and
    \begin{displaymath}
        J_A(a,\eta) = \frac{|\tau_1(a,\eta) \wedge \cdots \wedge \tau_{\vdim}(a,\eta)|}
        {|\zeta_1(a,\eta) \wedge \cdots \wedge \zeta_{\vdim}(a,\eta)|}
        \quad \text{for $(a,\eta) \in \widetilde{N}^{}(A)$}\,.
    \end{displaymath}
    If $W \subseteq N(A)$ is~$\HM^{\vdim}$~measurable and
    $\HM^{\vdim}(W) < \infty$, then for $\HM^{\vdim}$~almost all
    $(a,\eta) \in W$
    \begin{gather}
        \Tan^{\vdim}(\HM^{\vdim} \restrict W, (a,\eta))
        = \lin \bigl\{ \zeta_1(a,\eta), \ldots, \zeta_{\vdim}(a,\eta) \bigr\}
        \in \grass{2(\adim)}{\vdim} 
        \\
        \text{and} \quad
        (\HM^{\vdim} \restrict W, \vdim) \ap J_{\vdim} \pp(a,\eta)
        =  J_A(a,\eta) \CF{\widetilde{N}^{}_{\vdim}(A)}(a,\eta) \,.
    \end{gather}
\end{lemma}

\section{Quadratic flatness}
\label{sec:quadratic flatness}

The main goal of this section is to prove Theorem~\ref{lem:minkowski_content_of_nh_sets}. This result contains as a special case Theorem \ref{main:rectifiability}; cf. \ref{main_rectifiability proof}. 

\begin{lemma}[Basic estimates]
    \label{lem:one_sided_estimates}
    Suppose $\Omega \subseteq \R^{\adim}$ is open, $ A \subseteq \Omega $ is an $(\vdim, h)$~subset
    of $ \Omega $ with respect to~$ \phi $ and $\tau = (\tau_1, \ldots , \tau_{\vdim}) $ is
    a~(Euclidean) principle frame-field of~$A$. Then
    \begin{equation}
        \label{lem:one_sided_estimates_eq1}
        \textsum{i=1}{\vdim}  \kappa_{A,i}(a, \nu)\,\Der^2\phi(\nu)(\tau_i(a, \nu),\tau_i(a, \nu))
        \le  h
    \end{equation}
    and
    \begin{equation}
        \label{lem:one_sided_estimates_eq2}
        \textsum{i=1}{\vdim} \kappa_{A,i}(a, \nu)
        \le \vdim \kappa_{A,\vdim}(a, \nu)
        \le \vdim \gamma(\phi)^{-1} \bigl( h \phi(\nu)^{-1} + (\vdim-1) c_{2}(\phi)\ar{A}{}(a, \nu)^{-1} \bigr)
    \end{equation} 
    for $ \HM^{\vdim} $ almost all $(a, \nu) \in N(A) | \Omega  $. In particular, 
    \begin{equation}
         \label{lem:one_sided_estimates_eq3}
        \HM^{\vdim}\bigl((N(A)|\Omega) \without \widetilde{N}_{\vdim}(A) \bigr) =0. 
    \end{equation} 
\end{lemma}

\begin{proof}
    Choose $ (a, \nu) \in \widetilde{N}(A) | \Omega $ and $ 0 < r < \arl{A}{}(a, \nu) $. Without
    loss of generality we~may assume $a = 0$ and we set $T = \lin \{ \nu \}^{\perp}$. To~verify the
    first equality apply~\cite[3.29]{HugSantilli} with~$j=1$. Noting that
    $ r\nu \in \Unp(A) \cap S(A,r) $, we combine \cite[Lemma~2.41(e) and Lemma
    2.44]{Kolasinski2023b} to find $ \varepsilon, \delta > 0 $ and a semiconcave function
    $f : T \rightarrow \R $ such that
    \begin{gather}
        f(0) = r \,,
        \quad
        \Der f(0) = 0 \,,
        \quad
        \text{$f$ is pointwise twice differentiable at $0$} \,,
        \\
        \Der^2 f(0)(u, v)
        = - \Der \an{A}{}(r\nu) u \bullet v
        \qquad \text{for $ u,v \in \Tan(S(A, r), r\nu) = T $} \,,
        \\
        \cylind{\ker \beta(\nu)}{0}{\varepsilon}{\delta} \subseteq \Omega \,,
        \\
        S(A,r) \cap \cylind{T}{r\nu}{\delta}{\varepsilon} =
        \graph_{\nu}(f) \cap \cylind{T}{r\nu}{\delta}{\varepsilon} \,,
        \\
        \text{and} \quad
        \{ x: \da{A}{}(x) < r \} \cap \cylind{T}{r\nu}{\delta}{\varepsilon}
        = \cato_{\nu}(f) \cap \cylind{T}{r\nu}{\delta}{\varepsilon} \,.
    \end{gather}
    Let $ g : T \rightarrow \R $ be defined by $ g(\sigma) = f(\sigma)-r$ for
    $\sigma \in T$ and observe that
    \begin{displaymath}
        A \cap \cylind{ \ker \beta(\nu)}{0}{\varepsilon}{\delta} \subseteq \Clos{\cato_{\nu}(g)} \,;
    \end{displaymath}
    otherwise, if
    $ x \in A \cap \cylind{\ker \beta(\nu)}{0}{\varepsilon}{\delta} \without
    \Clos{\cato_{\nu}(g)} $ and $y = T\, x$, then
    \begin{gather}
        x = y + s\nu \quad \text{for some $0 < s < \infty$ with $g(y) < s$} \,,
        \quad
        f(y) < s+r \,,
        \\
        y + (s+r)\nu = x + r\nu \in \cylind{T}{r\nu}{\delta}{\varepsilon} \without \Clos{\cato_{\nu}(f)} \,,
        \quad \da{A}{}(x+r\nu) > r 
    \end{gather}
    but, since $x \in A$, there holds $\da{A}{}(x+r\nu) \le r$.
    
    Recall that $A \cap \cylind{\ker \beta(\nu)}{0}{\varepsilon}{\delta} $ is an $(\vdim,h)$~subset
    of~$\cylind{ \ker \beta(\nu)}{0}{\varepsilon}{\delta}$ with respect to~$\phi$ and we now know
    that
    $ A \cap \cylind{\ker \beta(\nu)}{0}{\varepsilon}{\delta} \subseteq \Clos{\cato_{\nu}(g)} $;
    hence, we conclude from~\ref{lem:weak_maximum_principle}
    \begin{equation}
        \label{eq:chi_bound}
        \textsum{i=1}{\vdim}  \rchi_{A,i}(a + r\nu)
        \,\Der^2\phi(\nu)(\tau_i(a, \nu),\tau_i(a, \nu)) \le h \,.
    \end{equation}
    Observe, recalling~\ref{rem:def:kappas}, that for $i \in \{1,2,\ldots,\vdim\}$ 
    \begin{displaymath}
        \rchi_{A,i}(a + r\nu) = \lim_{t \to \kappa_{A,i}(a, \nu)^{-}} \frac{t}{1 + r t}
        = \left\{
            \begin{aligned}
              &\frac{\kappa_{A,i}(a, \nu)}{1 + r \kappa_{A,i}(a, \nu)}
              && \text{if $\kappa_{A,i}(a, \nu) < \infty$} \,,
              \\
              &\frac 1r
              &&\text{if $\kappa_{A,i}(a, \nu) = \infty$} \,.
            \end{aligned}
        \right.
    \end{displaymath}
    Since the choice of $0 < r < \arl{A}{}(a, \nu)$ was arbitrary, we may pass
    to the limit $r \to 0^{+}$ and conclude from~\eqref{eq:chi_bound} that
    $ \kappa_{A,\vdim}(a, \nu) < \infty $ and
    \begin{displaymath}
        \textsum{i=1}{\vdim}  \kappa_{A,i}(a, \nu)\,\Der^2\phi(\nu)(\tau_i(a, \nu),\tau_i(a, \nu)) \le  h.
    \end{displaymath}

    In order to prove the second estimate we can further assume that
    $ \kappa_{A,\vdim}(a, \nu) > 0 $, otherwise the conclusion is trivially
    true. Since $ \kappa_{A,i}(a, \nu) \ge - \ar{A}{}(a, \nu)^{-1} $ for
    $i \in \{1,2,\ldots,\vdim\}$ by~\cite[Remark 3.8]{HugSantilli}, we estimate
    \begin{multline}
        \kappa_{A,\vdim}(a, \nu)
        \le \gamma(\phi)^{-1}\Der^2\phi(\nu)(\tau_{\vdim}(a, \nu),\tau_{\vdim}(a, \nu))\,\kappa_{A,\vdim}(a, \nu)
        \\
        \le \gamma(\phi)^{-1}\Big( h  - \textsum{i=1}{\vdim-1}\kappa_{A,i}(a, \nu)
        \,\Der^2\phi(\nu)(\tau_i(a, \nu),\tau_i(a, \nu))\Big)
        \\
        \le \gamma(\phi)^{-1}\Big( h + \ar{A}{}(a, \nu)^{-1}\textsum{i=1}{\vdim-1}
        \Der^2\phi(\nu)(\tau_i(a, \nu),\tau_i(a, \nu))\Big)
        \\
        \le \gamma(\phi)^{-1}\bigl( h  + (\vdim-1) c_{2}(\phi)\ar{A}{}(a, \nu)^{-1}\bigr) \,.
        \qedhere
    \end{multline}
\end{proof}

\begin{remark}\label{rem:coarea_formula_exhaust}
    This remark will be used in the proof of~\ref{lem:minkowski_content_of_nh_sets}. Its purpose
    should be clear given the fact that $N(A)$ might not have, a~priori, locally finite
    $\HM^{\vdim}$~measure. Nonetheless, $N(A)$ is always countably
    $(\HM^{\vdim}, \vdim)$-rectifiable, hence a~countable union of sets with locally finite
    $\HM^{\vdim}$~measure. Consequently, the coarea formula may be applied to $ N(A) $, as long as
    the integrand is non-negative. To make this statement rigorous suppose
    $\Omega \subseteq \R^{\adim}$ is open, $A \subseteq \Omega $ is relatively closed, and
    $ g : \R^{\adim} \times \sphere{\vdim} \rightarrow \R $ is a non-negative Borel function. Define
    \begin{displaymath}
        R_A(t) = N(A)|\Omega \cap \{ (a,\nu) : \ar{A}{}(a,\nu) > t \}
        \quad \text{for $0 < t < \infty$} \,.
    \end{displaymath}
    Using~\ref{lem:HS39} and its symbols together with~\cite[Theorem~3.16(a)]{HugSantilli} we see
    that for $ t > 0 $ the set $R_A(t)$ meets each compact subset of
    $ \R^{\adim} \times \sphere{\vdim} $ on an~$(\HM^{\vdim},\vdim)$-rectifiable set and
    \begin{displaymath}
        J_A(a,\nu)\CF{\widetilde{N}_{\vdim}(A)}(a,\nu) = \bigl(\HM^{\vdim} \restrict R_A(t) , \vdim \bigr) \ap J_{\vdim} \pp(a,\nu)
    \end{displaymath}
    for $ \HM^{\vdim} $ almost all $(a,\nu) \in R_A(t)$. Employing
    \cite[Lemma~3.25(a)(b)(c)]{HugSantilli}, we see that
    \begin{gather}
        \pp^{-1} \{a\} \cap \widetilde{N}_{\vdim}(A) =  \pp^{-1}\{a\} \cap N(A)
        \quad \textrm{for $ a \in \pp \lIm \widetilde{N}_{\vdim}(A) \rIm $}
        \\
        \text{and} \quad
        \HM^{\vdim} \bigl( \pp \lIm N(A) \rIm \without \pp \lIm \widetilde{N}_{\vdim}(A) \rIm \bigr) = 0 \,.
    \end{gather}
    Henceforth, we deduce by coarea formula~\cite[3.2.22]{Federer1969} that
    \begin{multline}
        \textint{\widetilde{N}_{\vdim}(A)\cap R_A(t)}{}
        J_A(a,\nu)\, g(a, \nu) \ud \HM^{\vdim}(a,\nu)
        \\
        = \textint{\pp \lIm \widetilde{N}_{\vdim}(A)\cap R_A(t) \rIm}{}
        \textint{\pp^{-1}\{a\} \cap R_A(t) \cap \widetilde{N}_{\vdim}(A)}{}
        g\ud \HM^0 \ud \HM^{\vdim}(a)
        \\
        = \textint{\pp \lIm \widetilde{N}_{\vdim}(A) \rIm}{}
        \textint{\pp^{-1}\{a\} \cap R_A(t)}{}
        g \ud \HM^0 \ud \HM^{\vdim}(a) \\
        = \textint{\Omega \cap \pp \lIm N(A) \rIm}{}
        \textint{\pp^{-1}\{a\} \cap R_A(t)}{}
        g \ud \HM^0 \ud \HM^{\vdim}(a)
        \quad \text{for $0 <  t < \infty$} \,.
    \end{multline}
    Using twice monotone convergence theorem, we deduce that
    \begin{multline}
        \textint{\widetilde{N}_{\vdim}(A)|\Omega}{}
        J_A(a,\nu)\, g(a, \nu) \ud \HM^{\vdim}(a,\nu)
        \\
        = \lim_{t \to 0^{+}}\textint{\widetilde{N}_{\vdim}(A) \cap R_A(t)}{}
        J_A(a,\nu)\, g(a, \nu) \ud \HM^{\vdim}(a,\nu)
        \\
        = \lim_{t \to 0^{+}}  \textint{\pp \lIm N(A)|\Omega \rIm}{}
        \textint{\pp^{-1}\{a\} \cap R_A(t)}{}
        g \ud \HM^0 \ud \HM^{\vdim}(a)
        \\
        = \textint{\pp \lIm N(A)|\Omega \rIm}{} \textint{N(A,a)}{}
        g(a, \nu) \ud \HM^0(\nu) \ud \HM^{\vdim}(a) \,.
    \end{multline}
    In particular, \emph{if $ B \subseteq \Omega $ is compact and
      $ \HM^{\vdim}(\pp \lIm N(A) \rIm \cap B) < \infty $}, we set $ g = \CF{B} \circ \pp$ and
    recall \cite[Lemma 3.25(c)]{HugSantilli} to conclude that
    \begin{multline}
        \textint{\widetilde{N}_{\vdim}(A)| B}{} J_A(a,\nu) \ud \HM^{\vdim}(a,\nu)
        = \textint{\pp \lIm N(A) \rIm \cap B}{} 
        \HM^0(N(A,a)) \ud \HM^{\vdim}(a)
        \\
        = \HM^{\vdim}\bigl( B \cap \{a : \HM^0(N(A,a)) =1\}\bigr)
        \\
        + 2\, \HM^{\vdim}\bigl(B \cap \{a : \HM^0(N(A,a)) = 2\}\bigr) < \infty.
    \end{multline}
\end{remark}

\begin{theorem}
    \label{lem:minkowski_content_of_nh_sets}
    Suppose $0 < h < \infty$, $\Omega \subseteq \R^{\adim}$ is open,  $A$ is an
    $(\vdim,h)$~subset of~$\Omega$ with respect to $ \phi $ such that
    \begin{gather}
        \label{lem:minkowski_content_of_nh_sets hypothesis}
        \LM^{\adim}(A) = 0
        \quad \text{and} \quad
        \HM^{\vdim} \restrict \pp[N(A)]\ \text{ is a Radon measure over $ \Omega $} \,,
        \\
        \text{and set} \quad
        S = A \cap \big\{ a :
        \exists \, \nu \in \sphere{\vdim} \, \textrm{s.t.} \,
        N(A,a) = \{ \nu,\, -\nu \} \big\} \,.
    \end{gather} 

    Then $ \HM^{\vdim} \restrict A $ is a Radon measure and $ \HM^\vdim(A \without S) =0 $. In particular, $ A $ is $(\HM^n,n)$-rectifiable of class $ \cnt{2} $.
\end{theorem}

\begin{proof}
         Define
        \begin{gather}
                \rho(a,\nu,r) = \inf\{r ,\, \ar{A}{}(a, \nu)\}
                \quad \text{for $(a,\nu) \in N(A)$ and $0 < r  < \infty$} \,,
                \\
                T = A \cap \bigl\{ a : \HM^{0}(N(A,a)) = 1 \bigr\} \,, \\
                 \textrm{and} \quad     f(a,\nu,s) = \textprod{j=1}{\vdim} \bigl( 1 + s \kappa_{A,j}(a,\nu) \bigr)
                \quad \textrm{for $(a,\nu, s) \in \widetilde{N}_{\vdim}(A) \times \R $}\,,
        \end{gather}
and notice that $ S =  A \cap \bigl\{ a : \HM^{0}(N(A,a)) = 2 \bigr\} $.
Suppose $ K \subseteq A $ is a compact set and,  for $ r > 0 $, define
    \begin{displaymath}
        K_r = K + \oball 0r = \R^{\adim} \cap \bigl\{x : \da{K}{}(x) < r \bigr\}
    \end{displaymath}
    and set
    \begin{displaymath}
        r_0 = \tfrac 12 \inf \bigl\{ |x-y| : x \in K ,\, y \in \R^{\adim} \without \Omega \bigr\} \,.
    \end{displaymath}
    Clearly, $ r_0 > 0 $.   Recall, by Lemma~\ref{lem:one_sided_estimates}, that
    \begin{equation}\label{eq:Lusin}
        \HM^{\vdim}\bigl( (N(A)|\Omega) \without \widetilde{N}_{\vdim}(A)\bigr) = 0 \,.
    \end{equation}  
    and deduce from \eqref{eq:Lusin}, from the estimates in Lemma~\ref{lem:one_sided_estimates}, and
    from~\cite[Remark~3.8]{HugSantilli} that there exists $0 < \Delta < \infty$ such that
    \begin{multline}
        \label{eq:bound_on_f}
        0 < f(a, \nu, t) \leq \Delta \bigl( 1 + t \bigr)^{\vdim}
        \\
        \text{for $ \HM^{\vdim} $ almost all $(a,\nu) \in N(A)| \Omega $ and all $0 < t < \ar{A}{}(a,\nu)$} \,.
    \end{multline}
    This is a crucial observation that allows to employ the Lebesgue dominated convergence theorem
    later on and exhibits the subtlety of dealing with~$(\vdim,h)$~sets.

    Note that $ \CF{K_r}(a + t \nu) = 0 $ whenever
    $ (a, \nu) \in N(A)| (\R^{\adim} \without \Omega) $ and $ 0 < t < r < r_0 $. Therefore,
    using~$ \LM^{\adim}(A) = 0 $, \cite[Corollary~3.18]{HugSantilli}, and~\eqref{eq:Lusin} we get
    for $ 0 < r < r_0 $
    \begin{multline}
        \LM^{\adim}(K_r) = \LM^{\adim}(K_r \without A)
        \\
        = \textsum{d=0}{\vdim} \textint{\widetilde{N}_d(A)}{}
        J_A(a, \nu) \textint{0}{\rho(a,\nu,r)} \CF{K_r}(a + t \nu)\, t^{\vdim-d} f(a, \nu, t)
        \ud \LM^{1}(t) \ud \HM^{\vdim}(a, \nu)
        \\
        = \textint{\widetilde{N}_{\vdim}(A)|\Omega}{}
        J_A(a, \nu) \textint{0}{\rho(a,\nu,r)}
        \CF{K_r}(a + t \nu)\, f(a, \nu, t) \ud \LM^{1}(t) \ud \HM^{\vdim}(a, \nu) \\
        = \textint{N(A)|\Omega}{}
        J_A(a, \nu) \textint{0}{\rho(a,\nu,r)}
        \CF{K_r}(a + t \nu)\, f(a, \nu, t) \ud \LM^{1}(t) \ud \HM^{\vdim}(a, \nu) \,.
    \end{multline}
    We compute the pointwise limit of the inner integral as a function over
    $\widetilde{N}_{\vdim}(A)|\Omega$. Noting that $ \CF{K_r}(a + t \nu) = 1 $ for
    $ (a, \nu) \in N(A)|K $ and $ 0 < t < \rho(a,\nu,r) $, we~obtain
    \begin{displaymath}
        \lim_{r \to 0} \tfrac{1}{2r} \textint{0}{\rho(a,\nu,r)}
        \CF{K_r}(a + t\nu) \, f(a,\nu,t) \ud \LM^{1}(t)
        = \tfrac 12
        \quad
        \text{for $ (a, \nu) \in \widetilde{N}_{\vdim}(A)|K $} \,.
    \end{displaymath}
    On the other hand,  
    \begin{gather}\label{eq1}
        \text{if $(a, \nu) \in N(A)| (A \without K)$ and $0 < t < r \leq \tfrac{1}{2}\da{K}{}(a)$,
          then $ \CF{K_r}(a + t \nu) =0 $} \,;
    \end{gather}
    hence, 
    \begin{displaymath}
        \lim_{r \to 0} \tfrac{1}{2r} \textint{0}{\rho(a,\nu,r)}
        \CF{K_r}(a + t\nu) \, f(a,\nu,t)\ud \LM^{1}(t) = 0
        \quad
        \text{for $ (a, \nu) \in \widetilde{N}_{\vdim}(A)| (\Omega \without K) $} \,.
    \end{displaymath}
    Recalling~\eqref{eq:bound_on_f} we get for $ \HM^{\vdim} $ almost all
    $(a, \nu)\in N(A)| \Omega $
    \begin{multline}
        \textint{0}{\rho(a,\nu,r)} f(a, \nu,t) \ud \LM^{1}(t)
        \leq \textint{0}{\rho(a,\nu,r)} \Delta \bigl( 1 + t \bigr)^{\vdim} \ud \LM^{1}(t)
        \\
        = \tfrac{\Delta}{\adim} \bigl( \bigl( 1 + \rho(a,\nu,r) \bigr)^{\adim} - 1 \bigr)
        \le \tfrac{\Delta}{\adim} \bigl( \bigl( 1 + r \bigr)^{\adim} - 1 \bigr) \,.
    \end{multline}
    Employing Remark~\ref{rem:coarea_formula_exhaust} and our
    hypothesis~\eqref{lem:minkowski_content_of_nh_sets hypothesis} yields
    \begin{displaymath}
        \textint{N(A)|B}{} J_A \ud \HM^{\vdim}  < \infty
        \quad \text{whenever $ B \subseteq \Omega $ is compact} \,.
    \end{displaymath}
    Therefore, we may apply the dominated convergence theorem and use again
    Remark~\ref{rem:coarea_formula_exhaust} to~get
    \begin{multline}
        \lim_{r \to 0^{+}} \tfrac{1}{2r} \textint{N(A)|K}{}
        J_A(a, \nu) \textint{0}{\rho(a,\nu,r)} \CF{K_r}(a + t \nu)\, f(a, \nu, t)
        \ud \LM^{1}(t)  \ud \HM^{\vdim}
        \\
        = \tfrac 12 \textint{N(A)|K}{} J_A(a, \nu) \ud \HM^{\vdim}(a, \nu) \\
        = \HM^{\vdim}( K \cap S ) + \tfrac 12 \HM^{\vdim}(K \cap T ) 
    \end{multline} 
    and
    \begin{multline}
        \limsup_{r \to 0^{+}} \tfrac{1}{2r}
        \textint{N(A)|(\Omega \without K)}{}
        J_A(a, \nu) \textint{0}{\rho(a,\nu,r)} \CF{K_r}(a + t \nu)\, f(a, \nu, t)
        \ud \LM^{1}(t)  \ud \HM^{\vdim}
        \\
        = \limsup_{r \to 0^{+}} \tfrac{1}{2r}
        \textint{N(A)|(K_\varepsilon \without K)}{}
        J_A(a, \nu) \textint{0}{\rho(a,\nu,r)} \CF{K_r}(a + t \nu)\, f(a, \nu, t)
        \ud \LM^{1}(t)  \ud \HM^{\vdim}
        =0
    \end{multline}
    for any $ 0 < \varepsilon <  2r_0  $. Clearly, $M_1,M_2 \subseteq \pp \lIm N(A) \rIm$, so we deduce that
    \begin{multline}
        \label{eq:minkowski_content_for_nh_sets} \noeqref{eq:minkowski_content_for_nh_sets} 
        \lim_{r \to 0^{+}} \tfrac{1}{2r} \LM^{\adim}(K_r)
        \\
        = \lim_{r \to 0^{+}} \tfrac{1}{2r} \textint{N(A)|K}{}
        J_A(a, \nu) \textint{0}{\rho(a,\nu,r)} \CF{K_r}(a + t \nu)\, f(a, \nu, t)
        \ud \LM^{1}(t) \ud \HM^{\vdim}(a, \nu)
        \\
        + \lim_{r \to 0^{+}} \tfrac{1}{2r} \textint{N(A)|(\Omega \without K)}{}
        J_A(a, \nu) \textint{0}{\rho(a,\nu,r)} \CF{K_r}(a + t \nu)\, f(a, \nu, t)
        \ud \LM^{1}(t) \ud \HM^{\vdim}(a, \nu)
        \\
        = \HM^{\vdim}( K \cap S ) + \tfrac 12 \HM^{\vdim}(K \cap T ) < \infty  \,.
    \end{multline}

  Since the (lower) Minkowski content can always be bounded from below, up to a~dimensional
      constant, by the Hausdorff measure (cf.\ \cite[pag.\ 79]{Mattila1995}), we see\footnote{In
        this paper the Hausdorff measure is defined as in \cite[2.10.2]{Federer1969}, which differs
        from the Hausdorff measure defined in \cite{Mattila1995} by a dimensional constant.} that
      \begin{displaymath}
          \tfrac{\unitmeasure{\adim}}{2^{n+1}\, \unitmeasure{\vdim}}\, \HM^{\vdim}(K)
          \leq \lim_{r \to 0^{+}} \tfrac{1}{2r} \LM^{\adim}(K_r)
          = \HM^{\vdim}( K \cap S ) + \tfrac 12 \HM^{\vdim}(K \cap T) < \infty
      \end{displaymath}
      for every compact set $ K \subseteq A $, henceforth $ \HM^\vdim \restrict A $ is a~Radon
      measure. In particular, $ \HM^\vdim(K) =0 $ whenever $ K \subseteq A \without (S \cup T) $ is
      compact, so we conclude from \cite[2.2.5]{Federer1969} that
      \begin{displaymath}
          \HM^\vdim(A \without(S \cup T)) =0 \,.
      \end{displaymath}
      Clearly $ T \subseteq \pp \lIm N(A) \rIm$ is countably $(\HM^{\vdim},\vdim)$~rectifiable.
      If~$ K \subseteq T $ is a compact set, then $ K $ is $(\HM^{\vdim}, \vdim) $-rectifiable and,
      employing the $ \vdim $-dimensional Gross measure (cf. \cite[2.10.4]{Federer1969}) and
      \cite[3.2.26, 3.2.37]{Federer1969}, we see that
      \begin{displaymath}
          \tfrac 12 \HM^{\vdim}(K) 
          = \liminf_{r \to 0^{+}} \tfrac{1}{2r} \LM^{\adim}(K_r)  
          \geq \GM^{\vdim}(K)   = \HM^{\vdim}(K) \,;
      \end{displaymath} 
      thus $ \HM^{\vdim}(K) =0 $, and we conclude that $ \HM^{\vdim}(T) =0 $ again by
      \cite[2.2.5]{Federer1969}. It follows that $ \HM^n(A \setminus S) =0 $ and $ A $ is $(\HM^n,n) $-rectifiable of class $ \cnt{2} $ by \cite{Menne2019a}.
\end{proof}

\begin{remark}
    \label{main_rectifiability proof} Theorem \ref{main:rectifiability} readily follows from Theorem
    \ref{lem:minkowski_content_of_nh_sets} since, by
    Lemma~\ref{lem:support_of_a_varifold_is_nh_set}, we see that $ \spt \| V \| $ is
    a~$(\vdim, h)$~subset of $ \Omega $ with respect to~$ \phi $.
\end{remark}

\begin{remark}\label{accumulating parallel lines}
    The following example shows the necessity of the hypothesis
    \ref{lem:minkowski_content_of_nh_sets}\eqref{lem:minkowski_content_of_nh_sets
      hypothesis}. Suppose $ \{ a_j : j \in \natp \} $ are decreasing positive numbers converging to
    $ 0 $ and set
    \begin{displaymath}
        A = \bigl( \{0\} \times \R \bigr)
        \cup \tbcup_{j=1}^{\infty} \bigl( \{a_j\} \times \R \bigr)
        \cup \bigl( \{-a_j\} \times \R \bigr) \,.
    \end{displaymath}
    Then $ A $ is a $(1,0)$-set of $ \R^2 $ with respect to the Euclidean norm; indeed
    (cf.~Lemma~\ref{lem:support_of_a_varifold_is_nh_set}) $ A $ is the support of the varifold
    \begin{displaymath}
        V = \textsum{j=1}{\infty} \tfrac{1}{2^j} \bigl(
        \var{1}(\{a_j\} \times \R) + \var{1}(\{-a_j\} \times \R)
        \bigr) \in \Var{1}(\R^2) \,,
    \end{displaymath}
    and $ \delta V =0 $. Clearly, $ \HM^1 \restrict \pp[N(A)] $ is \emph{not} a Radon measure over
    $ \R^2 $ and for each $ x \in \{0\} \times \R $ there exists no $ \nu \in \sphere{1} $ such that
    $ \da{A}{}(x + s \nu) = s $ for some $ s > 0 $.
\end{remark}

\section{Regularity}
\label{sec:regularity}

In this section we combine Theorem \ref{main:rectifiability} with Allard's regularity theorem
\cite[3.6]{Allard1986} to prove Theorem~\ref{main:regularity}.

\nsubsection{Proof of Theorem \ref{main:regularity}} Define $0 < \delta < 1$ as the number whose existence is guaranteed
by~\textsc{The~regularity theorem} of~\cite[p.~27]{Allard1986} with
\begin{align}
        &0, &&\tfrac 32 \unitmeasure{\vdim}, &&\tfrac 12, &&\tfrac 12, &&c_0(\phi), &&c_1(\phi),
        &&\gamma(\phi), &&c_2(\phi), &&c_3(\phi), &&\alpha
        \\
        \text{in place of} \quad
        &A, &&M, &&\lambda, &&\varepsilon, &&M_{0,0}, &&M_{0,1}
        &&\gamma, &&N_{0,2}, &&N_{0,3}, &&\alpha \,.
\end{align}
Define $ A = \spt \| V \| $ and
\begin{displaymath}
        S = \bigl\{
        x \in A
        : \exists \, r > 0 \ 
        \exists \, \nu \in \sphere{\vdim} \quad
        \oball{a+r\nu}{r} \cap A = \varnothing = \oball{a-r\nu}{r} \cap A
        \bigr\} \,.
\end{displaymath}
Since $ V \in \IVar{\vdim}(\Omega) $, we see by \cite[3.5]{Allard1972} that $ \| V \| = \theta \cdot \HM^\vdim \restrict W$, where
$ W\subseteq \Omega $ is a countably $(\HM^\vdim,\vdim)$-rectifiable set and $ \theta $ is a
$ \HM^\vdim \restrict W$-summable $\natp $-valued function. It follows from the absolute continuity hypothesis that $ \HM^\vdim(A \without W) =0 $, hence $ A $ is countably $(\HM^{\vdim}, \vdim)$-rectifiable
and $\HM^{\vdim} \restrict A \leq \| V \| $. In particular, $ \HM^\vdim \restrict A $ is a Radon measure
over~$ \Omega $ and we may apply Theorem \ref{main:rectifiability} to infer that
\begin{equation}
    \label{main:rectifiability_eq2}
    \HM^{\vdim} \bigl( A \without S \bigr) = 0 \,.
\end{equation}

We fix $ a \in S $ and $\nu \in \sphere{\vdim}$ such that $ \density^{\vdim}(\| V \|,a) = 1$ and
$N(A,a) = \{ \nu, -\nu \}$. Let $T \in \grass{\adim}{\vdim}$ satisfy $T \nu = 0$. Without loss of
generality we assume $ a =0 $. Notice that there exist and $0 < s < \infty$ such that 
\begin{gather}
        \bigl\{ x : |T(x)| \leq s ,\, |T^\perp(x)| \leq 2s \bigr\} \subseteq \Omega
        \\
        \text{and} \quad
        \label{eq:regularity:touching_with_balls}
        A \cap \oball{s \nu}{s}
        = A \cap \oball{- s \nu}{s} = \varnothing \,.
\end{gather}
We define $ Z = \bigl\{ x : | T(x)| < s/2, \; | T^\perp(x)| < 5s/2 \bigr\}$ and
$ Q = \{x : | T^\perp(x)| \leq s\} $ and set
\begin{displaymath}
        \widetilde{V} = V \restrict Z  \times \grass{\adim}{\vdim} \in \IVar{\vdim}(Z) \,.
\end{displaymath}
Notice that $ Q \cap T^{-1}(\oball{0}{s/2}) \subseteq Z $ and,
employing~\eqref{eq:regularity:touching_with_balls}, we observe that
\begin{displaymath}
        \spt \| \widetilde{V} \| \cap T^{-1} \lIm \oball 0r \rIm
        \subseteq \bigl\{ x : | T(x)| < r, \; |T^\perp(x)| < r^2/s \bigr\}
        \subseteq Q
        \quad \text{for $ 0 < r \leq \tfrac s2 $} \,.
\end{displaymath}
Since $ \density^{\vdim}(\| \widetilde{V} \|,0) = 1 $, we can choose $ 0 < r < \frac{s}{2}  $ so that 
\begin{gather}
        Hr \leq \delta, \quad
        r^2 \leq \delta \, 2s^2 / ( 3\unitmeasure{\vdim} ) \,,
        \\
        \text{and} \quad
        \tfrac 12\, \unitmeasure{\vdim} \rho^{\vdim}
        \leq \|\widetilde{V} \|\bigl(Z \cap T^{-1} \lIm \oball 0{\rho} \rIm \bigr)
        \leq \tfrac 32\, \unitmeasure{\vdim} \rho^{\vdim}
        \quad \text{for $ 0 < \rho  \leq r $} \,.
\end{gather}
In particular, we notice that 
\begin{displaymath}
        r^{-\vdim-2} \textint{T^{-1}\lIm \oball 0r \rIm}{} |T^\perp(z)|^2 \ud \| \widetilde{V} \|(z)
        \leq \delta \,.
\end{displaymath}
We conclude both \cite[2.2(1)-(8)]{Allard1986} and \cite[3.6(1)-(5)]{Allard1986} are satisfied and
we infer from the regularity theorem \cite[pag.\ 27]{Allard1986} that
$ \spt \| \widetilde{V} \| \cap T^{-1} \lIm \oball{0}{r/2} \rIm $ is the graph of a
$ \cnt{1, \alpha} $-function.

\subsection*{Acknowledgements}
Part of this work was done while the first author was hosted by the second author at University of
L'Aquila. The~research of S.\,K.\ was financed by the \href{https://ncn.gov.pl/}{National Science
  Centre Poland} grant number 2022/46/E/ST1/00328.
\\ The research of M.\ S.\ is partially supported by INDAM-GNSAGA and PRIN project 20225J97H5.

\bigskip

{\small \noindent
  Sławomir Kolasiński \\
  Uniwersytet Warszawski, Instytut Matematyki \\
  ul. Banacha 2, 02-097 Warszawa, Poland \\
  \texttt{s.kolasinski@mimuw.edu.pl}
}

\bigskip

{\small \noindent
  Mario Santilli \\
  Department of Information Engineering, Computer Science and Mathematics,\\
  Università degli Studi dell'Aquila\\
  via Vetoio 1, 67100 L’Aquila, Italy\\
  \texttt{mario.santilli@univaq.it}
}

\end{document}